\documentclass[10pt]{article}

 \usepackage{amsmath,amsthm,amssymb}
 \usepackage{makeidx,epsfig,lscape}
 \usepackage{color,colortbl}
 \usepackage{fancyhdr}
 \usepackage{times}
 \usepackage{xcolor,pict2e}
 \usepackage[latin1]{inputenc}
 
\newcommand{\R}{\mathbb R}

\newcommand{\E}{\mathbb E}

 \renewcommand{\headrulewidth}{0pt}
 \renewcommand{\footrulewidth}{0.5pt}

 \definecolor{myaqua}{rgb}{0.0,0.5,0.55}
 \definecolor{lightaqua}{rgb}{0.75,0.95,0.95}

 \usepackage[colorlinks = true,
            linkcolor = myaqua,
            urlcolor  = blue,
            citecolor = red,
            pdfpagemode=UseOutlines,
            bookmarksnumbered=true,
            pdfpagelabels,
            breaklinks]{hyperref}

\usepackage{caption}
\usepackage{floatrow}

\newtheorem{Theorem}{Theorem}
\newtheorem{Prop}{Proposition}
\newtheorem{Lemma}{Lemma}

\def\lin#1#2{\textcolor[rgb]{0.6,0.6,0.6}{\vspace*{#1mm} \hrule
   height 3 pt \vspace*{#2mm}}}
\def\bt{\begin{tabular}}
\def\et{\end{tabular}}
\def\and{\mbox{ and }}
\def\E{\mbox{\bf E}}

\def\1{{\bf 1}}

 \def\boxx#1#2#3#4#5{
 {\linethickness{#4pt}\put(#1,#5){\color{myaqua}{\line(1,0){#3}}}}
 \multiput(#1,#2)(0,#4){2}{\line(1,0){#3}}
 \multiput(#1,#2)(#3,0){2}{\line(0,1){#4}}
  }

\begin{document}
	$\mbox{ }$

 \vskip 12mm

{ 

{\noindent{\Large\bf\color{myaqua}
  Variable selection in discriminant analysis for mixed variables and several groups  }}
%
\\[6mm]
{\bf Alban Mbina Mbina$^{1,a}$, Guy Martial Nkiet$^{1,b}$ Fulgence EYI-OBIANG$^{1,c}$}}
\\[2mm]
{ 
 $^1$Laboratoire URMI, Universit\'e des Sciences et Techniques de Masuku, BP 943 Franceville, Gabon.\\
 Email: $^a${\color{blue}{\underline{\smash{albanmbinambina@yahoo.fr}}}}, $^b${\color{blue}{\underline{\smash{gnkiet@hotmail.com}}}}, $^c${\color{blue}{\underline{\smash{feyiobiang@yahoo.fr}}}}
\lin{5}{7}




{  
 {\noindent{\large\bf\color{myaqua} Abstract}{\bf \\[3mm]
 \textup{
 We propose a method for  variable selection in  discriminant analysis  with mixed categorical and continuous variables. This method is based on a criterion that  permits to reduce the variable selection problem to a problem of estimating  suitable permutation and dimensionality.  Then, estimators for these parameters are proposed and the resulting method for selecting variables is shown to be consistent. A simulation study that permits to  study several poperties of the proposed  approach  and to compare it with an existing method is given. 
 }}}
 \\[4mm]
 {\noindent{\large\bf\color{myaqua} Keywords:}{\bf \\[3mm]
 Variable selection; Discriminant analysis; Classification; Mixed variables
}}\\[4mm]{\noindent{\large\bf\color{myaqua} MSC:}{\color{blue} 62H30; 62H12}}
\lin{3}{1}

\renewcommand{\headrulewidth}{0.5pt}
\renewcommand{\footrulewidth}{0pt}

 \pagestyle{fancy}
 \fancyfoot{}
 \fancyhead{} 
 \fancyhf{}
 \fancyhead[RO]{\leavevmode \put(-90,0){\color{myaqua}} \boxx{15}{-10}{10}{50}{15} }
 \fancyfoot[C]{\leavevmode
 \put(-2.5,-3){\color{myaqua}\thepage}}

 \renewcommand{\headrule}{\hbox to\headwidth{\color{myaqua}\leaders\hrule height \headrulewidth\hfill}}

}}

\section{Introduction}
\noindent
The problem of classifying an observation into one of several classes in the basis of data consisting of both continuous and categorical variables is an old problem that  have been tackled under different forms in the literature.  The earliest works in this field go back to Chang and Afifi (1974) and Krzanowski (1975) who used the location model introduced by Olkin and Tate (1961) to form a classification rule in the context of discriminant analysis involving two groups. More recent work has focused on defining distance measures between populations or making inference  on them  (e.g., Krzanowski 1983, Krzanowski  1984, Bar-Hen and Daudin 1995, Bedrick et al. 2000, de Leon and Carri\`ere 2005). One of the most important problem in the context described above is the problem of selecting  the  appropriate categorical and/or continuous variables to use for discrimination. Indeed, it is well recognized that using fewer variables improve classification performance and permits to avoid estimation problems (e.g. McLachlan 1992, Mahat et al. 2007). There are several works dealing with this problem, mainly in the context of location model. Some of these works are based on the use of distances between populations for determining the most predictive variables (Krzanowski 1983, Daudin 1986, Bar-Hen and Daudin 1995, Daudin and Bar-Hen 1999). Krusinska (1989a, 1989b, 1990) used methods based on the percentage of missclassification, Hotelling's $T^2$ and graphical models. More recently, Mahat et al. (2007) proposed a method based on distance between groups as measured by smoothed Kullback-Leiler divergence. All these works consider the case of two groups and, to the best of our knowledge, the case of more than two groups have not yet  been considered for variable selection purpose. So, it is of great interest to introduce a method that can be used when the number of groups is greater than two. Such an approach have been proposed recently in Nkiet (2012) for the case of continuous variables only. It is based on a criterion that permits to characterize the set of variables that are appropriate for discrimination by means of two parameters, so that the variable selection problem reduces to that of estimating these parameters. 

In this paper, we extend  the approach of Nkiet (2012) to the case of mixed variables. The resulting method has two advantages; first, it can be used when the number of groups is greater than two, and secondly it just require that the random vector consisting of the continuous variables has finite fourth order moment. No assumption on the distribution of this random vector is needed and, therefore, we do not suppose that the location model holds. In section 2, we introduce a criterion by means of which the set of  variables to be estimated is characterized by means of suitable permutation and dimensionality. Then, estimating this criterion is tackled in section 3. More precisely, empirical estimators as well as non-parametric smoothing procedure are used for defining an estimator of the criterion. In the first case, we obtain properties of the resulting estimator that permits to obtain its asymptotic distribution. Section 4  is devoted to the definition of our proposal for variable selection. Consistency of the method, when empirical estimators are used,  is then proved. Section 5  is devoted to the presentation of numerical experiments made in order to study several properties of the proposal and   to compare it with an existing method.    The first issue that is adressed  concerns the impact of chosing penalty functions that are involved in our procedure, and that of the type of estimators that is used. The results reveal low impact on the performance of the proposed method. Since this method depends on two real  parameters, it is of interest to study their influence on its performance and, consequently, to define a strategy that permits to chose optimal values for them. The simulation results clearly show their impact on the performance, and we propose a method based on leave-one-out cross validation for obtaining optimal results.  When using this appoach, the obtained  results show  that the proposal is competitive with that of Mahat et al. (2007). All the proofs are  given in Section 6.

\section{Statement of the problem}
\label{sec2}
Letting $\left(\Omega,\mathcal{A},P\right)$ be a probability space, we consider random vectors $$X=\left(X^{(1)},\cdots,X^{(p)}\right)^T \,\,\, \textrm{and}\,\,\, Y=\left(Y^{(1)},\cdots,Y^{(d)}\right)^T$$  defined on this probability space and valued into $\mathbb{R}^p$ and $\{0,1\}^d$ respectively. The r.v. $X$ consists of continuous random variables whereas  $Y$ consists of binary random variables. As usual, $Y$ may be associated to  a multinomial random variable by considering  $U=1+\sum_{j=1}^dY^{(j)}\,2^{j-1}$ which is valued into $\{1,\cdots,M\}$, where $M=2^d$. Suppose that the observations of $(X,Y)$ come from $q$ groups $\pi_1,\cdots\pi_q$  (with $q\geq 2$) characterized by a random variable $Z$ valued into $\{1,\cdots,q\}$; this means that $(X,Y)$ belongs to  $\pi_\ell$  if, and only if, one has $Z=\ell$.   Such framework has  been considered in the literature for classification purpose. Indeed, for the case of two groups, that is  when $q=2$, Krzanowski (1975) proposed a classification rule  based on the location model  under the assumption that the distribution of $X$, conditionally to $U=m$ and $Z=\ell$, is the multivariate normal distribution $N(\mu_{m\ell},\Sigma)$. This rule allocates a future observation $(x,y)$ of $(X,Y)$ to $\pi_1$ if 
\begin{equation}\label{fisher}
\left(\mu_{m1}-\mu_{m2}\right)^T\Sigma^{-1}\left(x-\frac{1}{2}(\mu_{m1}+\mu_{m2})\right)\geq\log\left(\frac{p_{m2}}{p_{m1}}\right)+\log(\alpha),
\end{equation}
where $m=1+\sum_{j=1}^dy^{(j)}\,2^{j-1}$,  $p_{m\ell}=P(U=m|Z=\ell)$ and $\alpha$ is a constant that depends on costs due to missclassification and prior probabilities for the two groups. The case where $q>2$ was considered  by  de Leon et al (2011) in the context of  general mixed-data models. In this case, the optimum rule classifies an observation  $(x,y)$ as belonging to $\pi_{\ell^\ast}$ if 
\begin{equation}\label{classmult}
\delta_m^{(\ell^\ast)}(x,y)=\max_{\ell =1,\cdots,q}\delta_m^{(\ell)}(x,y) 
\end{equation}
where
\[
\delta_m^{(\ell)}(x,y)=(\mu_{m\ell})^T\Sigma^{-1}x-\frac{1}{2}(\mu_{m\ell})^T\Sigma^{-1}\mu_{m\ell}+\log(p_{m\ell})+\log(\beta_\ell),
\]
with $\beta_\ell=P(Z=\ell)$. As it can be seen,  these rules involve observations of all  the variables $X^{(j)}$ in $X$. Nevertheless, as it is well recognized (see, e.g., McLachlan 1992, Mahat et al. 2007), using fewer variables improve  classification performance. So, it is of real interest to perform selection of the $X^{(j)}$'s from a sample of $(X,Y,Z)$. For doing that,  we extend to the case of mixed variables  an approach used  in Nkiet (2012) for the case of continuous variables only. This approach first  consists in  introducing a criterion by means of which the set of variables that are adequate for discrimination is characterized. From now on, we assume that $\mathbb{E}\left(\Vert X\Vert^4\right)<+\infty$, where $\Vert\cdot\Vert$denotes the usual Euclidean norm of $\mathbb{R}^p$, and  for  any $m\in\{1,\cdots,M\}$  we consider
\[
p_m=P(U=m),\,\,\mu_{m} = \E\left(X|U=m\right)\,\,\textrm{ and }\,\, \mu_{\ell,m} = \E\left(X|Z=\ell,U=m\right).
\]
We denote by $\otimes $ the tensor product defined as follows: for any $(u,v)\in (\mathbb{R}^p)^2$, $u\otimes v$ is the linear map $h\in\mathbb{R}^p\mapsto \langle u,h \rangle v$, where $\langle\cdot , \cdot>$ is the usual Euclidean inner product of $\mathbb{R}^p$. Then, we consider the conditional covariance operator given by
\[
V_{m} = \E\left(\left(X-\mu_{m}\right)\otimes\left(X-\mu_{m}\right)|U=m\right)
\]
and assume that it is invertible. Let us put $I:=\{1,\cdots,p\}$ and, for any subset $K$ of $ I$:
\begin{equation}\label{qk}
Q_{K|m} := A^{*}_{K}\left( A_{K}V_m A^{*}_{K} \right)^{-1}A_{K}
\end{equation}
where  $A_{K}$ denotes the projector:
\[
x = (x_{i})_{i \in I}\,\in \mathbb{R}^{p} \mapsto x_{K} = (x_{i})_{i \in K}\in \mathbb{R}^{\textrm{card}(K)}.
\] 
Then, putting 
\[
p_{\ell|m}=P(Z=\ell |U=m),
\]  
we introduce
\begin{equation}\label{critm}
\xi_{K|m} =  \sum_{\ell=1}^{q} p_{\ell|m}^{2}\| \left(I_{\R^{p}} - V_{m} Q_{K|m}\right)\left(\mu_{\ell,m} - \mu_{m}\right) \|^2,
\end{equation} 
where  $I_{\R^{p}}$ is the identity operator of $\mathbb{R}^p$. This leads us to consider the criterion
\begin{equation}\label{crit}
\xi_{K} =  \sum_{m=1}^{M} p_{m}^{2}\xi_{K|m} 
\end{equation}
by means of which we will try to characterize the subset of variables that is adequate for discrimination. For defining this subset of $I$, let us keep in mind that if  a variable $X^{(j)}$ is not adequate for discrimination, then it is the case whatever  $Y$. So, 
denoting by $I_{0,m}$ the subset of $I$ consisting of continuous  variables in $X$  that are not adequate for discrimination  when $U=m$, we can consider the subset 
\[
I_{0}  = \bigcap_{m=1}^{M}I_{0,m}
\]
as the subset of variables that are not adequate for discrimination in the mixed case. Therefore, our problem reduces to the problem of estimating the subset $I_1$ given by
\[
I_{1}  =I-I_0=  \bigcup_{m=1}^{M}I_{1,m}
\]
where  $I_{1,m}=I-I_{0,m}$. As it was done in Nkiet (2012) for the case of continuous variables only, an explicit expression of  $I_{1,m}$ can be obtained by using results from McKay (1977). Indeed, let  $\lambda_{1,m} \geq \lambda_{2,m} \geq \cdots \geq \lambda_{p,m}$ denote the eigenvalues of $T_{m} = V_{m}^{-1}B_{m}$ where $B_m$ is the between groups  covariance operator conditionally to $U=m$  given by
\[
B_{m} = \sum_{\ell=1}^{q} p_{\ell|m}(\mu_{\ell,m} - \mu_{m})\otimes(\mu_{\ell,m} - \mu_{m}), 
\]
and let  $\upsilon_{i}^{m}=(\upsilon_{i1}^{m},\cdots,\upsilon_{ip}^{m})^{T}$ $(i=1,\cdots,p)$ be an eigenvector of $T_{m}$ associated with $\lambda_{i,m}$. Then, 
$I_{1,m} = \left\{k \in I | \exists i\in\{1,\cdots, r_m\}, \upsilon_{ik}^{m} \neq 0  \right\}$. Now, we are able to give a characterization of $I_1$ by means of the criterion given in (\ref{crit}).

\begin{Prop}\label{prop1}
For $ K \subset I$, we have $\xi_{K} = 0$ if  and only if $I_{1}\subset  K$.
\end{Prop}
From this proposition, it is easily seen that, putting $K_i=I-\{i\}$, one has the equivalence:  $\xi_{K_i}>0\Leftrightarrow
i\in I_{1}$.   Now, let  $\sigma$ be the permutation  of $I$ such that:

\bigskip

\textit{(A1)} \ $\xi_{K\sigma\left(  1\right)  }\geq\xi_{K\sigma\left(
2\right)  }\geq\cdots\geq\xi_{K\sigma\left(  p\right)  };$

\bigskip

\textit{(A2)} \ $\xi_{K\sigma\left(  i\right)  }=\xi_{K\sigma\left(  j\right)
}$ and $i<j\;$imply $\sigma\left(  i\right)  <\sigma\left(  j\right)  .$

\bigskip

\noindent Since  $I_{1}$ is a non-empty set, there
exists an integer $s\in I$ which is equal to $p$ when $I_{1}=I$, and
satisfying
\[
\xi_{K\sigma\left(  1\right)  }\geq\cdots\geq\xi_{K\sigma\left(  s\right)
}>\xi_{K\sigma\left(  s+1\right)  }=\cdots=\xi_{K\sigma\left(  p\right)  }=0
\]
when $I_{1}\neq I$. Hence
\begin{equation*}\label{carac}
I_{1}=\left\{
\sigma\left(  i\right)  ;\;1\leq i\leq s\right\}.
\end{equation*}
Therefore, estimating $I_1$ reduces to estimating the two parameters $\sigma$ and $s$. For doing that, we first need to consider an estimator of the criterion given in (\ref{crit}).
\section{Estimating the  criterion}
\label{sec3.2}

Let  $\{(X_i,Y_i,Z_i)\}_{1\leq i\leq n}$ be an i.i.d. sample of $(X,Y,Z)$ with 
\[
X_i=(X^{(1)}_i,\cdots,X^{(p)}_i)^T\,\,\textrm{ and }\,\,  Y_i=(Y^{(1)}_i,\cdots,Y^{(d)}_i)^T;
\] 
we put $U_i=1+\sum_{j=1}^dY^{(j)}_i\,2^{j-1}$. In this section, we define estimators for the criterion given in (\ref{crit}) by estimating the parameters involved in its definiion. First, empirical estimators are introduced and properties of the resulted estimator of the criterion are given, and secondly we consider estimators obtained by using non-parametric smoothing procedures as in Mahat et al. (2007).
\subsection{Empirical estimators}
\label{EE}
Putting  
\[
\widehat{N}^{(n)}_{m} = \sum_{i=1}^{n} \textrm{\textbf{1}}_{\left\{U_{i}=m \right\}}\,\textrm{  and  }\, \widehat{N}^{(n)}_{\ell,m} = \sum_{i=1}^{n} \textrm{\textbf{1}}_{\left\{ Z_{i}=\ell,U_{i}=m \right\}},
\]
we estimate  $p_{m}$, $p_{\ell|m}$, $\mu_{m}$,  $\mu_{\ell,m}$  and  $V_{m}$ respectively by:
\[
\widehat{p}^{(n)}_{m} = \frac{\widehat{N}^{(n)}_{m}}{n},\,\,\,\,  \widehat{p}^{(n)}_{\ell|m} = \frac{\widehat{N}^{(n)}_{\ell,m}}{\widehat{N}^{(n)}_{m}},\,\,\,\, \widehat{\mu}^{(n)}_{m} = \frac{1}{\widehat{N}^{(n)}_{m}}\sum_{i=1}^{n}\textrm{\textbf{1}}_{\left\{U_{i}=m \right\}}X_{i},
\]
\[
\widehat{\mu}^{(n)}_{\ell,m} = \frac{1}{\widehat{N}^{(n)}_{\ell,m}}\sum_{i=1}^{n}\textrm{\textbf{1}}_{\left\{Z_i=\ell,U_{i}=m \right\}}X_{i}
\]
and
\[
\widehat{V}^{(n)}_{m} = \frac{1}{\widehat{N}^{(n)}_{m}}\sum_{i=1}^{n}\textrm{\textbf{1}}_{\left\{U_{i}=m \right\}}(X_{i}-\widehat{\mu}^{(n)}_{m})\otimes(X_{i}-\widehat{\mu}^{(n)}_{m}).
\]
Then, considering 
\[
\widehat{Q}^{(n)}_{K|m} = A^{*}_{K}\left( A_{K}\widehat{V}^{(n)}_{m} A^{*}_{K} \right)^{-1}A_{K}
\]
and
\[
\widehat{\xi}^{(n)}_{K|m} =  \sum_{\ell=1}^{q}(\widehat{p}^{(n)}_{\ell|m})^2\| \left(I_{\R^{p}} -\widehat{V}^{(n)}_{m} \widehat{Q}^{(n)}_{K|m}\right)\left(\widehat{\mu}^{(n)}_{\ell,m}  - \widehat{\mu}^{(n)}_{m}\right) \|^2,
\] 
we take as estimator of  $\xi_K$ the random variable $\widehat{\xi}^{(n)}_K$ defined by:
\begin{equation}\label{estimcrit}
\widehat{\xi}^{(n)}_{K}  =  \sum_{m=1}^{M}  \left(\widehat{p}^{(n)}_{m}\right)^{2}\widehat{\xi}^{(n)}_{K|m}. 
\end{equation}

\noindent Now, we will give a result which establishes strong consistency for $\widehat{\xi}_{K}^{(n)}$  and will be useful for determining its asymptotic distribution and consistency of the proposed method for selecting variables.
Let us  consider the random variables
\[ 
\mathcal{X} =
\begin{pmatrix} \textrm{\textbf{1}}_{\left\{Z=1,U=1\right\}}X & \dots & \textrm{\textbf{1}}_{\left\{Z=1,U=M\right\}}X\\ 
                \textrm{\textbf{1}}_{\left\{Z=2,U=1\right\}}X & \dots & \textrm{\textbf{1}}_{\left\{Z=2,U=M\right\}}X\\
                \vdots  & \ddots & \vdots\\ 
                \textrm{\textbf{1}}_{\left\{Z=q,U=1\right\}}X & \dots & \textrm{\textbf{1}}_{\left\{Z=q,U=M\right\}}X\\
\end{pmatrix},
\]
\[
\mathcal{X}_{i} =
\begin{pmatrix} \textrm{\textbf{1}}_{\left\{Z_i=1,U_i=1\right\}}X_i & \dots & \textrm{\textbf{1}}_{\left\{Z_i=1,U_i=M\right\}}X_i\\ 
                \textrm{\textbf{1}}_{\left\{Z_i=2,U_i=1\right\}}X_i & \dots & \textrm{\textbf{1}}_{\left\{Z_i=2,U_i=M\right\}}X_i\\
                \vdots  & \ddots & \vdots\\ 
                \textrm{\textbf{1}}_{\left\{Z_i=q,U_i=1\right\}}X_i & \dots & \textrm{\textbf{1}}_{\left\{Z_i=q,U_i=M\right\}}X_i\\
\end{pmatrix}
\]
valued into the space $\mathcal{M}_{qp,M}(\mathbb{R})$ of $pq\times M$ matrices. We also introduce the random vectors
\[ 
\mathcal{Y} =
\begin{pmatrix} \textrm{\textbf{1}}_{\left\{U=1\right\}}X\\ 
                \textrm{\textbf{1}}_{\left\{U=2\right\}}X\\
                \vdots \\ 
                \textrm{\textbf{1}}_{\left\{U=M\right\}}X\\
\end{pmatrix},\,\,
\mathcal{Y}_{i} =
\begin{pmatrix} \textrm{\textbf{1}}_{\left\{U_i=1\right\}}X_i\\ 
                \textrm{\textbf{1}}_{\left\{U_i=2\right\}}X_i\\
                \vdots\\ 
                \textrm{\textbf{1}}_{\left\{U_i=M\right\}}X_i\\
\end{pmatrix},\,  \mathcal{Z} =
\begin{pmatrix} \textrm{\textbf{1}}_{\left\{U=1\right\}}\\ 
                \textrm{\textbf{1}}_{\left\{U=2\right\}}\\
                \vdots \\ 
                \textrm{\textbf{1}}_{\left\{U=M\right\}}\\
\end{pmatrix},\, \mathcal{Z}_{i} =
\begin{pmatrix} \textrm{\textbf{1}}_{\left\{U_i=1\right\}}\\ 
                \textrm{\textbf{1}}_{\left\{U_i=2\right\}}\\
                \vdots\\ 
                \textrm{\textbf{1}}_{\left\{U_i=M\right\}}\\
\end{pmatrix}, 
\]
the random variables valued into  $\mathcal{M}_{q,M}(\mathbb{R})$,

\[
\mathcal{U} =
\begin{pmatrix} \textrm{\textbf{1}}_{\left\{Z=1,U=1\right\}} & \dots & \textrm{\textbf{1}}_{\left\{Z=1,U=M\right\}}\\ 
                \textrm{\textbf{1}}_{\left\{Z=2,U=1\right\}} & \dots & \textrm{\textbf{1}}_{\left\{Z=2,U=M\right\}}\\
                \vdots  & \ddots & \vdots\\ 
                \textrm{\textbf{1}}_{\left\{Z=q,U=1\right\}} & \dots & \textrm{\textbf{1}}_{\left\{Z=q,U=M\right\}}\\
\end{pmatrix},\,\,
\mathcal{U}_{i} =
\begin{pmatrix} \textrm{\textbf{1}}_{\left\{Z_i=1,U_i=1\right\}} & \dots & \textrm{\textbf{1}}_{\left\{Z_i=1,U_i=M\right\}}\\ 
                \textrm{\textbf{1}}_{\left\{Z_i=2,U_i=1\right\}} & \dots & \textrm{\textbf{1}}_{\left\{Z_i=2,U_i=M\right\}}\\
                \vdots  & \ddots & \vdots\\ 
                \textrm{\textbf{1}}_{\left\{Z_i=q,U_i=1\right\}} & \dots & \textrm{\textbf{1}}_{\left\{Z_i=q,U_i=M\right\}}\\
\end{pmatrix},
\]
and the random variables 
\[ 
\mathcal{V}  =
\begin{pmatrix} \textrm{\textbf{1}}_{\left\{U=1\right\}}X\otimes X\\ 
                \textrm{\textbf{1}}_{\left\{U=2\right\}}X\otimes X\\
                \vdots \\ 
                \textrm{\textbf{1}}_{\left\{U=M\right\}}X\otimes X\\
\end{pmatrix},\,\,
 \mathcal{V} _{i} =
\begin{pmatrix} \textrm{\textbf{1}}_{\left\{U_i=1\right\}}X_{i}\otimes X_i\\ 
                \textrm{\textbf{1}}_{\left\{U_i=2\right\}}X_{i}\otimes X_i\\
                \vdots\\ 
                \textrm{\textbf{1}}_{\left\{U_i=M\right\}}X_{i}\otimes X_i\\
\end{pmatrix}
\]
valued into $(\mathcal{L}(\mathbb{R}^p))^M$, where  $\mathcal{L}(\mathbb{R}^p)$ denotes the space of  operators from  $\mathbb{R}^p$ to itself. Further, we consider the random variables
\[
\mathcal{W}=(\mathcal{X},\mathcal{Y},\mathcal{Z},\mathcal{U},\mathcal{V}),\,\,
\mathcal{W}_i=(\mathcal{X}_i,\mathcal{Y}_i,\mathcal{Z}_i,\mathcal{U}_i,\mathcal{V}_i)
\]
valued into the Euclidean vector space  
\[
\mathcal{E} = \mathcal{M}_{qp,M}(\R)\times \R^{pM}\times  \R^{M}\times  \mathcal{M}_{q,M}(\R)\times \mathcal{L}(\R^{p})^{M};
\] 
then, we put
\begin{equation}\label{wn}
\widehat{W}^{(n)}= \sqrt{n}\left(\frac{1}{n}\sum_{i=1}^{n}\mathcal{W}_{i} - \mathbb{E}(\mathcal{W}) \right).
\end{equation}
Note that, for any $(a,b,c,d,e)$ $\in$ $\mathcal{E}$,  we  can write:

\[ a =
\begin{pmatrix} a_{1,1}  & a_{1,2}  & \dots & a_{1,M}\\
                a_{2,1} & a_{2,2} & \dots & a_{2,M}\\
                \vdots & \vdots & \ddots & \vdots\\ 
                a_{q,1} & a_{q,2} & \dots & a_{q,M}\\
\end{pmatrix}, \textrm{ where } a_{\ell,m} \in \mathbb{R}^{p},\,\, 1\leq \ell \leq q , \,\, 1\leq m \leq M,
\]
\[ b =
\begin{pmatrix} b_{1}\\
                b_{2}\\
                \vdots\\ 
                b_{M}\\
\end{pmatrix}, \textrm{ where }   b_{m} \in \R^{p}, 1\leq m \leq M,
\]
\[ 
c =
\begin{pmatrix} c_{1}\\
               c_{2}\\
                \vdots\\ 
               c_{M}\\
\end{pmatrix}, \textrm{ where } \,\, c_{m} \in \mathbb{R},\,\, 1\leq m \leq M,
\]
\[ d =
\begin{pmatrix} d_{1,1}  & d_{1,2}  & \dots & d_{1,M}\\
                d_{2,1} & d_{2,2} & \dots & d_{2,M}\\
                \vdots & \vdots & \ddots & \vdots\\ 
                d_{q,1} & d_{q,2} & \dots &d_{q,M}\\
\end{pmatrix}, \textrm{ where }  c_{\ell,m} \in \mathbb{R}, \,\,1\leq \ell \leq q,  \,\,1\leq m \leq M,
\]

\[ e =
\begin{pmatrix} e_{1}\\
                e_{2}\\
                \vdots\\ 
                e_{M}\\
\end{pmatrix}, \textrm{ where }  e_{m} \in \mathcal{L}( \mathbb{R}^{p}), 1\leq m \leq M.
\]
We introduce the  projectors
\begin{eqnarray*}
\pi_{1}^{\ell m} &\,\,:\,\,& (a,b,c,d,e) \in\mathcal{E}  \mapsto  a_{\ell , m} \in  \mathbb{R}^{p}, \\
\pi_{2}^{m} &:& (a,b,c,d,e) \in \mathcal{E}  \mapsto b_{m} \in  \mathbb{R}^{p}, \\
\pi_{3}^{m} &:& (a,b,c,d,e) \in \mathcal{E}\mapsto c_{m} \in  \mathbb{R}, \\
\pi_{4}^{\ell m} &:& (a,b,c,d,e) \in\mathcal{E} \mapsto d_{\ell ,m} \in  \mathbb{R}, \\
\pi_{5}^{m} &:& (a,b,c,d,e) \in \mathcal{E}\mapsto e_{m} \in \mathcal{L}( \mathbb{R}^{p}), 
\end{eqnarray*}
the vector
\[
\Delta_{\ell,K|m}=\left(I_{\mathbb{R}^p}-V_mQ_{K|m}\right)\left(\mu_{\ell ,m}-\mu_m\right)
\]
and the   operators $\Lambda_{K|m}$, $\Phi_{\ell,K|m}$ and  $\Psi_{\ell,K|m}$ defined on $\mathcal{E}$ by
\[
\Lambda_{K|m}(T)= 2p_{m}\pi_{3}^{m}(T)\,\xi_{K|m},
\]
\[
\Phi_{\ell,K|m}(T)= p_{m}^{-1}\left(\pi_{4}^{\ell m}\left(T\right) - \pi_{3}^{m}\left(T\right)p_{\ell|m}\right)\|\Delta_{\ell,K|m}\|,
\]
and
\begin{eqnarray*}
\Psi_{\ell,K|m}\left(T\right) &=&p_m^{-1} \left(I_{\R^{p}} - V_{m}Q_{K,m}\right)
 \left[p_{\ell|m}^{-1}\left(\pi_{1}^{\ell m}(T) - \pi_{4}^{\ell m}(T)\mu_{\ell,m}\right)\right.
 - \pi_{2}^{m}(T) + \pi_{3}^{m}(T)\mu_{m}\\
& & -\left(\pi_{5}^{m}(T) - \pi_{3}^{m}(T)\left(V_{m} + \mu_{m}\otimes \mu_{m}\right)\right)Q_{K|m}\left(\mu_{\ell,m} - \mu_{m}\right)\\
& & +\left(\left(\pi_{2}^{m}(T) - \pi_{3}^{m}(T) \mu_{m}\right)\otimes \mu_{m} \right) Q_{K|m}\left(\mu_{\ell,m} - \mu_{m}\right)\\
& &+ \left.\left(\mu_{m}\otimes \left(\pi_{2}^{m}(T) - \pi_{3}^{m}(T)\mu_{m}\right)\right)Q_{K|m}\left(\mu_{\ell,m} - \mu_{m}\right) \right].
\end{eqnarray*}
Then,  we obtain the following theorem that gives  consistency for   $\widehat{\xi}_{K}^{(n)}$ and a  results that permits  to derive its  asymptotic distribution, and that is  useful for proving consistency of the proposed method for variable selection.

\begin{Theorem}\label{thr1}
For any subset  $K$ of  $ I$,
\begin{enumerate}
	\item[(i)]  $\widehat{\xi}_{K}^{(n)}$ converges almost surely to $\xi_{K}$ as $n \rightarrow +\infty$.
  \item[(ii)]   
\begin{eqnarray*}
  n \widehat{\xi}_{K}^{(n)} &=& \sum_{m=1}^{M}\sqrt{n}\widehat{\Lambda}_{K|m}^{(n)}(\widehat{W}^{(n)}) + \sum_{m=1}^{M}\sum_{\ell=1}^{q}\left(p_{m}\widehat{\Phi}_{\ell,K|m}^{(n)}\left(\widehat{W}^{(n)}\right)  \right.\\
                            &+& \left.
p_{m}p_{\ell|m}\|\widehat{\Psi}_{\ell,K|m}^{(n)}\left(\widehat{W}^{(n)}\right) + \sqrt{n}\Delta_{\ell,K\vert m}\| \right)^{2}\\
\end{eqnarray*}
\end{enumerate}
where $\left(\widehat{\Lambda}_{K|m}^{(n)}\right)_{n \in \mathbb{N}^\ast}$, $\left(\widehat{\Phi}_{\ell,K|m}^{(n)}\right)_{n \in \mathbb{N}^\ast}$ and $\left(\widehat{\Psi}_{\ell,K|m}^{(n)}\right)_{n \in \mathbb{N}^\ast}$  are  sequences  of random operators that  converge  almost surely uniformly to $\Lambda_{K|m}$, $\Phi_{\ell,K|m}$ and $\Psi_{\ell,K|m}$ respectively.
\end{Theorem}

\subsection{Non-parametric smoothing procedure}
\label{NE}

As it is well known, empirical estimators could be not suitable, because many cell incidences could be very small or zero, so that corresponding estimates could be poor or non-existent (see Aspakourov and Krzanowski 2000).  To overcome these problems, we propose in this section a non-parametric smoothing procedure for estimating the criterion (\ref{crit}). This procedure is based on smoothing as it is done in Aspakourov and Krzanowki (2000) and Mahat et al. (2007). Denote by $D$ the dissimilarity  defined on $\{1,\cdots,M\}^2$ by
\[
D(m,k)=\Vert \textrm{\textbf{x}}_m-  \textrm{\textbf{x}}_k\Vert^2,
\]
where, for any $k\in\{1,\cdots,M\}$, $\textrm{\textbf{x}}_k=\left(\textrm{\textbf{x}}_k^{(1)},\cdots,\textrm{\textbf{x}}_k^{(d)}\right)^T\in\{0,1\}^d$  is  the vector of binary variables satisfying $1+\sum_{j=1}^d\textrm{\textbf{x}}_k^{(j)}\,2^{j-1}=k$.  Then, given a smoothing parameter $\lambda\in]0,1[$, we consider the weights $w(m,k)=\lambda^{D(m,k)}$  and estimate  $p_{m}$, $p_{\ell|m}$, $\mu_{m}$,  $\mu_{\ell,m}$  and  $V_{m}$ respectively by:
\[
\widetilde{p}^{(n)}_{m} = \frac{\sum_{j=1}^Mw(m,j) \widehat{N}^{(n)}_{j}}{\sum_{k=1}^M\sum_{j=1}^Mw(m,j) \widehat{N}^{(n)}_{j}},\,\,\,\,  \widetilde{p}^{(n)}_{\ell\vert m} = \frac{\sum_{j=1}^Mw(m,j) \widehat{N}^{(n)}_{\ell , j}}{\widetilde{p}^{(n)}_{m}\sum_{k=1}^M\sum_{j=1}^Mw(m,j) \widehat{N}^{(n)}_{\ell ,j}},
\]
\[
\widetilde{\mu}^{(n)}_{m} = \left\{ \sum_{j=1}^{M} w(m,j)\widehat{N}^{(n)}_{j} \right\}^{-1} \sum_{j=1}^{M}\left\{ w(m,j)\sum_{n=1}^{n} X_i\textrm{\textbf{1}}_{\{U_i=j\}} \right\}, 
\]
\[
\widetilde{\mu}^{(n)}_{\ell ,m} = \left\{ \sum_{j=1}^{M} w(m,j)\widehat{N}^{(n)}_{\ell , j} \right\}^{-1} \sum_{j=1}^{M}\left\{ w(m,j)\sum_{n=1}^{n} X_i\textrm{\textbf{1}}_{\{Z_i=\ell, U_i=j\}} \right\}
\]
and
\[
\widetilde{V}^{(n)}_{m} = \left\{ \sum_{j=1}^{M} w(m,j)\widehat{N}^{(n)}_{ j} \right\}^{-1} \sum_{j=1}^{M}\left\{ w(m,j)\sum_{n=1}^{n} \textrm{\textbf{1}}_{\{ U_i=j\}}  (X_i-\widetilde{\mu}^{(n)}_{m}) \otimes (X_i-\widetilde{\mu}^{(n)}_{m})\right\}.
\]
Then, we obtain an estimator $\widetilde{\xi}_K^{(n)}$ of the criterion by replacing in (\ref{qk}), (\ref{critm}) and (\ref{crit}) the parameters  $p_{m}$, $p_{\ell|m}$, $\mu_{m}$,  $\mu_{\ell,m}$  and  $V_{m}$ by their estimators given above.

\section{Selection of variables}
\label{sec3.3}
Estimation of  $I_{1}$
reduces to that of $\sigma$ and $s$. In this section, estimators for these two parameters are
proposed. When empirical estimators are used, we then obtain consistency properties for these estimators.
\subsection{Estimation of $\sigma$ and $s$}
\label{subsec:estimpar}
Let us consider a sequence $\left(  f_{n}\right)  _{n\in\mathbb{N}^{\ast}}$ of
\ functions from $I$ \ to $\mathbb{R}_{+}$ such that  $f_n\sim  n^{-\alpha}f$ where
$\alpha\in\left]  0,1/2\right[  $ and $f$ is  a strictly decreasing function from
$I$ to $\mathbb{R}_{+}$.
Then, recalling that $K_{i}=I-\left\{  i\right\}  $, we put
\[
\widehat{\phi}_{i}^{\left(  n\right)  }=\widehat{\xi}_{K_{i}}^{\left(
n\right)  }+f_{n}\left(  i\right)  \text{\ \ \ \ (}i\in I\text{)}
\]
and we take as estimator of $\sigma$ the random permutation $\widehat{\sigma
}^{\left(  n\right)  }$ of $I$ such that

\[
\widehat{\phi}_{\widehat{\sigma}^{\left(  n\right)  }\left(  1\right)
}^{\left(  n\right)  }\geq\widehat{\phi}_{\widehat{\sigma}^{\left(  n\right)
}\left(  2\right)  }^{\left(  n\right)  }\geq\cdots\geq\widehat{\phi
}_{\widehat{\sigma}^{\left(  n\right)  }\left(  p\right)  }^{\left(  n\right)
}
\]
and if  $\widehat{\phi}_{\widehat{\sigma}^{\left(  n\right)  }\left(
i\right)  }^{\left(  n\right)  }=\widehat{\phi}_{\widehat{\sigma}^{\left(
n\right)  }\left(  j\right)  }^{\left(  n\right)  }$  with  $i<j$,
then $\widehat{\sigma}^{\left(  n\right)  }\left(  i\right)
<\widehat{\sigma}^{\left(  n\right)  }\left(  j\right)$.
Furthermore, we consider the random set  $\widehat{J}_{i}^{\left(  n\right)  }=\left\{
\widehat{\sigma}^{\left(  n\right)  }\left(  j\right)  ;\;1\leq j\leq
i\right\}  $ and the random variable
\[
\widehat{\psi}_{i}^{\left(  n\right)  }=\widehat{\xi}_{\widehat{J}%
_{i}^{\left(  n\right)  }}^{\left(  n\right)  }+g_{n}\left(  \widehat{\sigma
}^{\left(  n\right)  }\left(  i\right)  \right)  \;\;\;\;\;\text{(}i\in
I\text{)}
\]
where $\left(  g_{n}\right)  _{n\in\mathbb{N}^{\ast}}$ is a sequence of
\ functions from $I$ to \ $\mathbb{R}_{+}$ such that  $g_n\sim n^{-\beta}g$ where  $\beta
\in\left]  0,1\right[  $ and $g$ is  a strictly increasing function.
Then, we take as estimator of $s$ the random variable
\[
\widehat{s}^{\left(  n\right)  }=\min\left\{  i\in I\;/\;\widehat{\psi}%
_{i}^{\left(  n\right)  }=\min_{j\in I}\left(  \widehat{\psi}_{j}^{\left(
n\right)  }\right)  \right\}  .
\]
 The variable
selection is achieved by taking the random set
\[
\widehat{I}_{1}^{\left(  n\right)  }=\left\{  \widehat{\sigma}^{\left(
n\right)  }\left(  i\right)  \;;\;1\leq i\leq\widehat{s}^{\left(  n\right)
}\right\}
\]
  as estimator of $I_{1}$.

\subsection{Consistency}
\label{subsec:cons}

When  the empirical estimators defined in section \ref{EE} are considerd, we   establish  consistency for the preceding estimators. We first  give a  proposition that is useful  for proving the consistency theorem.  There exist
$t\in I$ and $\left(  m_{1},\cdots,m_{t}\right)  \in I^{t}$ such that
\ $m_{1}+\cdots+m_{t}=p$, and
$\xi_{K_{\sigma (  1)  }}    =\cdots=\xi_{K_{\sigma(
m_{1})  }}>\xi_{K_{\sigma (  m_{1}+1)  }}=\cdots
=\xi_{K_{\sigma (  m_{1}+m_{2})  }}>\cdots
\cdots   >\xi_{K_{\sigma (  m_{1}+\cdots+m_{t-1}+1)  }}=\cdots
=\xi_{K_{\sigma (  m_{1}+\cdots+m_{t})  }}$.
Then considering \ the set $E=\left\{  \ell\in\mathbb{N}^{\ast};1\leq \ell\leq
t,m_{\ell}\geq2\right\}  $ and putting $m_{0}:=0$ and $F_{\ell}:=\left\{\sum_{k=0}^{\ell-1}m_k+1,\cdots,\sum_{k=0}^{\ell}m_k-1\right\}  $
\ ($\ell\in\left\{  1,\cdots,t\right\}  $), we have:

\begin{Prop}\label{prop2}
If  $E \neq \varnothing$, then for any  $\ell \in E$ and  any $i \in F_{\ell}$, the sequence $n^{\alpha}\left(\widehat{\xi}^{(n)}_{K_{\sigma(i)}} - \widehat{\xi}^{(n)}_{K_{\sigma(i+1)}}\right)$ converges in probability  to $0$, as $n \rightarrow +\infty$.
\end{Prop}

\bigskip

\noindent From a proof that uses the preceding proposition and  that is similar to the proof  of Theorem 3.1 in Nkiet (2012), we obtain the consistency theorem given below. 

\begin{Theorem}\label{thrconv}
We have: 

(i) $\lim_{n\rightarrow+\infty}P\left(  \widehat{\sigma}^{\left(
n\right)  }=\sigma\right)  =1;$

(ii) $\widehat{s}^{\left(  n\right)  }$  converges in
probability to $s$,  as $n\rightarrow+\infty$.
\end{Theorem}

\bigskip

\noindent As a consequence of this theorem, we easily obtain:
\[
\lim_{n\rightarrow+\infty
}P\left(  \widehat{I}_{1}^{\left(  n\right)  }=I_{1}\right)  =1. 
\]
This shows the consistency of our method for selecting variables in discriminant analysis with mixed variables.

\section{Numerical experiments}
\label{sim}
In this section, we report results of  simulations made for studying  properties of the proposed method. Several issues are adressed: the influence of the penalty functions,    the type of estimator and  the parameters $\alpha$ and $\beta$ on the performance of the  procedure, optimal choice of these parameters  and comparison with the  method proposed in Mahat et al. (2007). Since this latter method  consider only the case of two groups, we  have placed ourselves in this framework although our method can be used for more than two groups. Each data set was generated as follows: $X_i$ is generated from a multivariate normal distribution in $\mathbb{R}^5$ with mean $\mu$ and covariance matrix  given by $\Gamma =\frac{1}{2}(I_5+J_5)$, where $I_5$ is the $5$-dimensional  identity matrix and $J_5$ is the $5\times 5$ matrix whose elements are all equal to $1$; $U^{(i)}$ is generated from the uniform distribution in $\{1,\cdots,8\}$, that is equivalent to generate $Y_i$ as random vector with $3$ coordinates being binary random variables. Two groups of data was generated as indicated above with $\mu=\mu_1=(0,0,0,0,0)^T$ for the first group (for which $Z_i=1$), and   $\mu=\mu_2=(1/4,0,1/2,0,3/4)^T$ for the second group (for which $Z_i=2$) .   Our simulated data is based on two independent data sets: training data and test data, each with sample size $n=100,\, 200,\,300,\,400,\,500$ and with size $n_1=n_2=n/2$ for the two groups.  The training data is used for selecting variables  and the test data is used for computing the classification capacity (CC), that is the proportion of correct classification (obtained by using the rule (\ref{fisher})) after variable selection. The average of CC over 1000 independent replications is used for measuring the performance of the methods.

\subsection{Influence  of penalty functions and type of estimator}

In order to evaluate the impact of penalty functions on the performance of our method we took $f_n(i)=n^{-1/4}/ h_k(i)$ and  $g_n(i)=n^{-1/4} h_k(i)$, $k=1,\cdots ,13$, with  $h_1(x)=x$, $h_2(x)=x^{0.1}$,   $h_3(x)=x^{0.5}$,  $h_4(x)=x^{0.9}$,   $h_5(x)=x^{10}$,  $h_6(x)=\ln(x)$,   $h_7(x)=\ln(x)^{0.1}$,  $h_8(x)=\ln(x)^{0.5}$,  $h_9(x)=\ln(x)^{0.9}$,  $h_{10}(x)=x\ln(x)$, $h_{11}(x)=\left(x\ln(x)\right)^{0.1}$, $h_{12}(x)=\left(x\ln(x)\right)^{0.5}$, $h_{13}(x)=\left(x\ln(x)\right)^{0.9}$. For each of these functions, we computed classifcation capacity by using both  empirical estimators and non-parametric smoothing procedure introduced in section \ref{NE}. For this latter type of estimator, the smoothing parameter $\lambda$ was computed from a cross validation method on the training sample in order to maximize classification capacity.  The results are given in Table 1. Comparing the results it is observed that  there is no significant difference between the results obtained for the different functions, and also between the two types of estimators. So, it seems that    chosing    penalty functions have no influence on the performance of our method. 

\bigskip

\begin{table}
\centering \caption{Average of classification capacity (CC)  over 1000 replications with different penalty functions  $f_n=n^{-1/4}/ h_k$ and  $g_n=n^{-1/4} h_k$, $k=1,\cdots,13$, and  $\alpha=0.25$, $\beta=0.5$.}
{\setlength{\tabcolsep}{0.08cm} 
\renewcommand{\arraystretch}{1} 
\begin{center}
{\begin{tabular}{ccccccccccccccc}
\hline
  &  &&& Empirical Estimator &&& Non-parametric Estimator\\
\hline
&\textbf{}&          \\
 n = 100(n$_{1}$=n$_{2}$=50) && Function &&CC &&& CC \\       
\hline 
\hline
 && $h_{1}$ && 0.60800 &&& 0.60800  \\ 
 && $h_{2}$ && 0.60900 &&& 0.60900  \\  
 && $h_{3}$ && 0.61000 &&& 0.61000 \\  
 && $h_{4}$ && 0.60900 &&& 0.60900  \\
 && $h_{5}$ && 0.60800 &&& 0.60800 \\
 && $h_{6}$ && 0.61028 &&& 0.61028  \\
 && $h_{7}$ && 0.60903 &&& 0.60903  \\
 && $h_{8}$ && 0.61066 &&& 0.61066  \\
 && $h_{9}$ && 0.60898 &&& 0.60898  \\
 && $h_{10}$ && 0.60842 &&& 0.60842  \\
 && $h_{11}$ && 0.60948 &&& 0.60948  \\
 && $h_{12}$ && 0.60915 &&& 0.60915  \\
 && $h_{13}$ && 0.60908 &&& 0.60908  \\
\hline 
 n = 300(n$_{1}$=n$_{2}$=150) && Function &&  &&&  &&&   \\
\hline
\hline
 && $h_{1}$ && 0.56421 &&& 0.56424  \\ 
 && $h_{2}$ && 0.56328 &&& 0.56332 \\  
 && $h_{3}$ && 0.56417 &&& 0.56417  \\  
 && $h_{4}$ && 0.56346 &&& 0.56346  \\
 && $h_{5}$ && 0.56382 &&& 0.56382 \\
 && $h_{6}$ && 0.56343 &&& 0.56343  \\
 && $h_{7}$ && 0.56374 &&& 0.56374  \\
 && $h_{8}$ && 0.56354 &&& 0.56354  \\
 && $h_{9}$ && 0.56312 &&& 0.56312  \\
 && $h_{10}$ && 0.56379 &&& 0.56379  \\
 && $h_{11}$ && 0.56404 &&& 0.56404  \\
 && $h_{12}$ && 0.56345 &&& 0.56345  \\
 && $h_{13}$ && 0.56375 &&& 0.56375  \\
\hline
 n = 500(n$_{1}$=n$_{2}$=250) && Function && &&&  &&& \\
\hline
\hline
 && $h_{1}$ && 0.54998 &&& 0.54998  \\ 
 && $h_{2}$ && 0.55047 &&& 0.55067  \\  
 && $h_{3}$ && 0.55032 &&& 0.55039  \\  
 && $h_{4}$ && 0.55049 &&& 0.55058 \\ 
 && $h_{5}$ && 0.54982 &&& 0.55982  \\ 
 && $h_{6}$ && 0.55058 &&& 0.55062  \\
 && $h_{7}$ && 0.55050 &&& 0.55054  \\
 && $h_{8}$ && 0.55008 &&& 0.55013  \\
 && $h_{9}$ && 0.54972 &&& 0.54934  \\
 && $h_{10}$ && 0.55013 &&& 0.55018  \\
 && $h_{11}$ && 0.55036 &&& 0.55037  \\
 && $h_{12}$ && 0.55046 &&& 0.55046  \\
 && $h_{13}$ && 0.55028 &&& 0.55028  \\
\hline
\end{tabular}}         
\end{center}}
\label{table:tab1}
\end{table}

\subsection{Influence of parameters $\alpha$ and $\beta$}
Since tuning parameters may have impact on the performance of a statistical procedure, it is important to study their influence. That is why  numerical experiments was made  in order to appreciate the influence of $\alpha$ and $\beta$ on the performance of our method. For doing that, we made  simulations as indicated above by taking  
\begin{equation}\label{penal}
f_n(i)=n^{-\alpha}/ h_7(i)\,\,\,\textrm{ and }\,\,\,  g_n(i)=n^{-\beta} h_7(i)
\end{equation} 
with $\alpha=0.1$, $0.2$, $0.3$, $0.4$, $0.45$, and $\beta$ varying in $[0,1[$.  The results are reported  in Fig. 1 (a)--(c) and show that the parameters $\alpha$ and $\beta$ clearly have impact on the performance of the method. Indeed, the obtained curves vary as $\alpha$ varies. Further,  for a fixed $\alpha$, CC  is not constant as  $\beta$ varies  in $[0,1[$. Finally, choosing the proper values for $\alpha$ and $\beta$  is an  important issue to address  in practice. An approach for doing that via cross-model validation  is described in the following section.
\bigskip

\begin{figure}[h!]
\centering
\includegraphics[width=8cm,height=8cm]{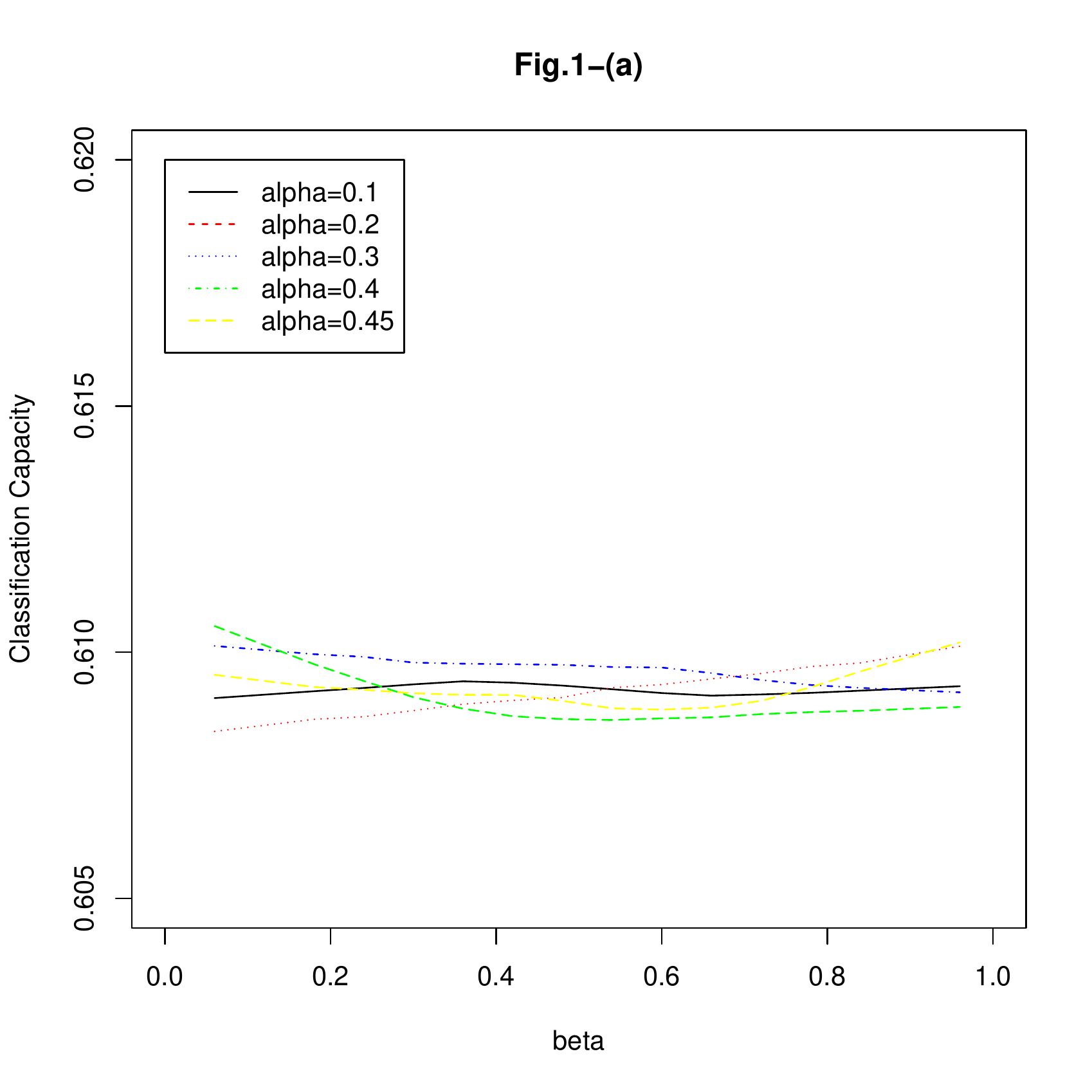}
\end{figure}
\begin{figure}[h!]
\centering
\includegraphics[width=8cm,height=8cm]{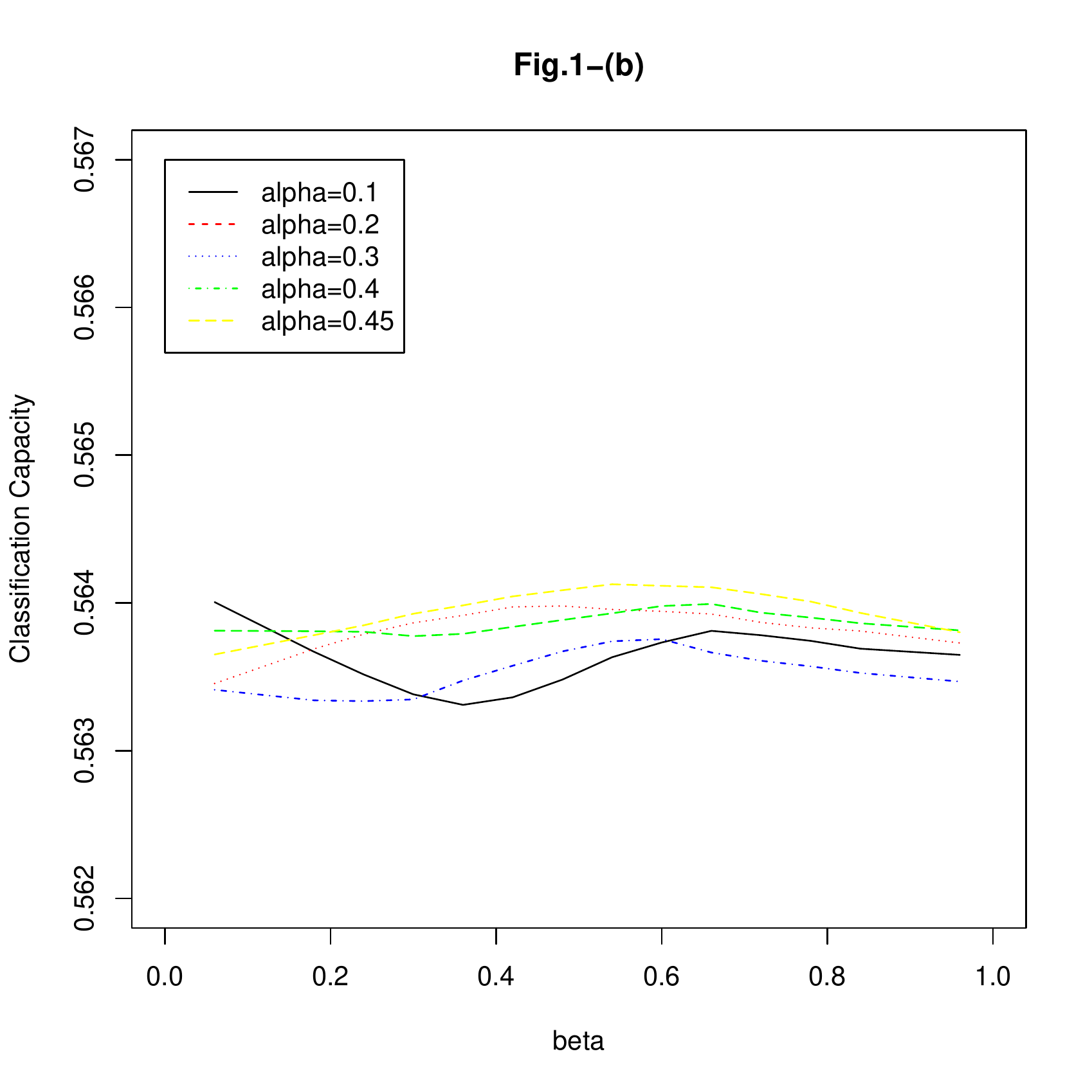}
\end{figure}
\begin{figure}[h!]
\centering
\includegraphics[width=8cm,height=8cm]{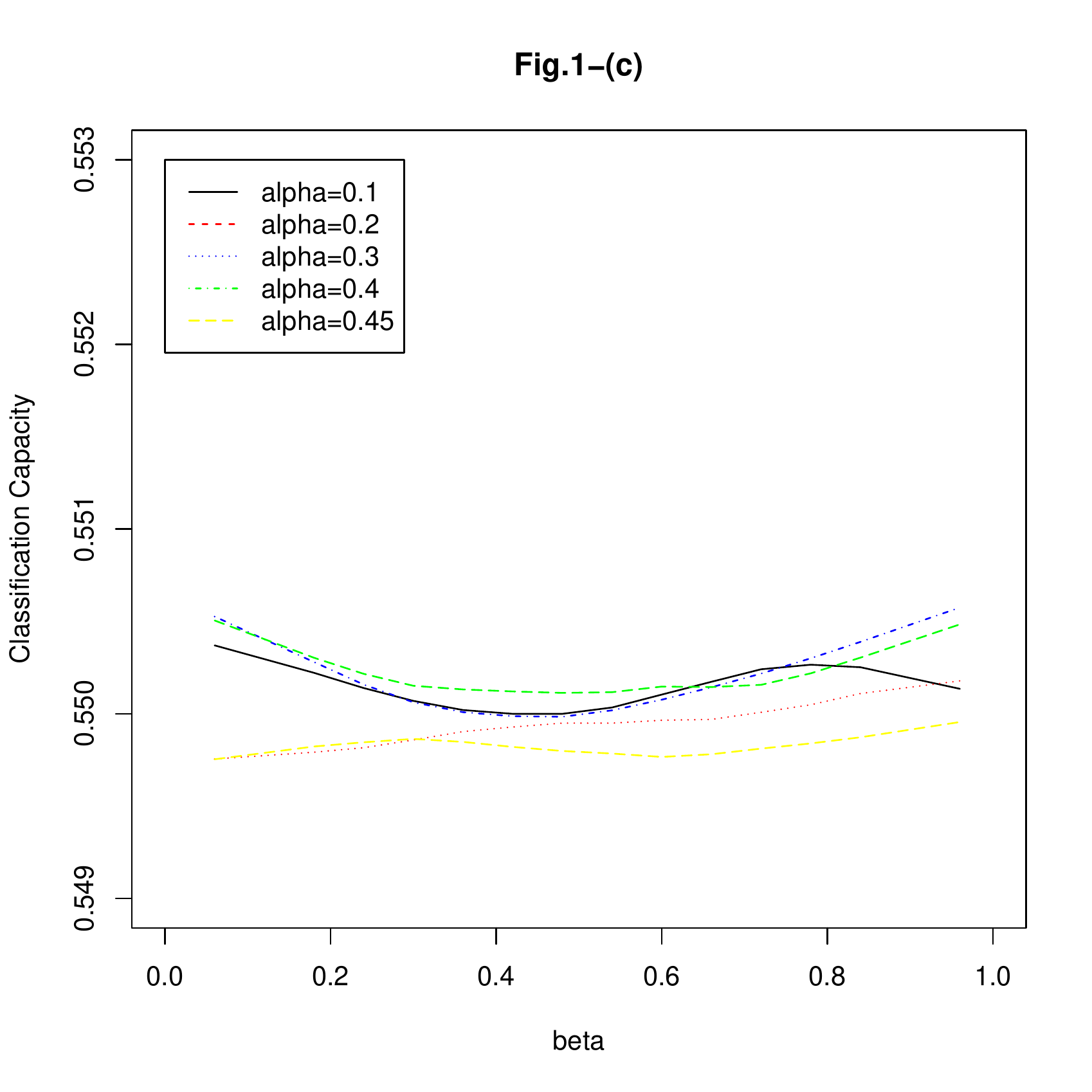}
\caption{Average of CC  over 1000 replications versus $\beta$, for different values of $\alpha$. (a) $n=100$  (b) $n=300$  (c) $n=500$.}
\end{figure}
\subsection{Chosing optimal $(\alpha,\beta)$}
We propose a method for making an optimal choice of  $(\alpha,\beta)$ based on leave-one-out cross validation used in order to minimize classification capacity. For each $k\in\{1,\cdots,n\}$, after removing the $k$-th observation for $X$ and $Y$   in the training sample  our method for selecting variable is applied on this remaining sample with a given value for $(\alpha,\beta)$ and penalty functions taken as in (\ref{penal}). Then, the observation that have been removed is allocated to a  group $\widetilde{g}_{\alpha,\beta}(k)$  in $\{1,\cdots,q\}$ by using the rule given in (\ref{fisher}) (for the two groups case) or in (\ref{classmult}) (for the case of more than two groups) based on the variables that have been selected in the previous step.  Then, we consider
\[
CV(\alpha,\beta)=\frac{1}{n}\sum_{k=1}^n\textrm{\textbf{1}}_{\{Z_k=\widetilde{g}_{\alpha,\beta}(k)\}}
\]
and we take as optimal value for $(\alpha,\beta)$ the pair $(\alpha_{opt},\beta_{opt})$ defined by:
\begin{equation}\label{opti}
(\alpha_{opt},\beta_{opt})=\underset{(\alpha,\beta)\in ]0,1/2[\times ]0,1[}{\mathrm{argmax}}CV(\alpha,\beta).
\end{equation}
\subsection{Comparison with the method of Mahat et al. (2007)}
For each of the 1000 independent replications: 
\begin{itemize}
\item[(i)] the training sample is used for selecting variables from our method and that of Mahat et al. (2007); our method is used with penalty functions given in (\ref{penal})  and  optimal  $(\alpha,\beta)$ obtained  by using leave-one-out cross validation as indicated in (\ref{opti});
\item[(ii)]the test sample is then used for computing classification capacity  for the two methods.
\end{itemize}
The average of  classification capacity over the 1000 replications is then computed. The results, given in Table 2,  do not show a superiority of one of the methods compared to the other.

\bigskip

\begin{table}
{\setlength{\tabcolsep}{0.08cm} 
\renewcommand{\arraystretch}{1} 
\begin{center}
{\begin{tabular}{ccccccccccccccc}
\hline
 &  & &  & &  & \\ 
  &  & &  & & CC & \\ 
\cline{4-7}
 $n$  & $n_{1}=n_{2}$ & & Our method & & & Mahat method\\      
\hline 
\hline
100 & 50 & & 0.60860 & & & 0.60826 \\   
200 & 100 & & 0.57810 & & & 0.57836\\  
300 & 150 & & 0.56244 & & & 0.56221 \\
400 & 200 & & 0.55534 & & & 0.55546 \\
500 & 250 & & 0.54989 & & & 0.54950\\
\hline\hline
\end{tabular}}         
\end{center}}
\centering \caption{\label{table:tab2} Average of classification capacity (CC)   over 1000 replications.}
\end{table}


\section{Proofs }
\label{proof}
 \subsection{Proof of Proposition 1}

For any fixed $m\in\{1,\cdots,M\}$, we denote by $P^{\left\{U=m\right\}}$ the  conditional probability to the event $\{U=m\}$. Then applying Theorem 2.1 of Nkiet (2012)  with the probability space $\left(\Omega,\mathcal{A},P^{\{U=m\}}\right)$, we obtain the equivalence:  $\xi_{K|m} = 0\Leftrightarrow  I_{1,m} \subset K$. Thus: 
\begin{eqnarray*}
\xi_{K} = 0 &\Leftrightarrow& \sum_{m=1}^{M}p_{m}^{2} \xi_{K|m} = 0 \\
            &\Leftrightarrow& \forall m \in \left\{1,\cdots,M\right\}, \xi_{K|m} = 0\\
  &\Leftrightarrow& \forall m \in \left\{1,\cdots,M\right\},  I_{1,m} \subset K\\
&\Leftrightarrow&\bigcup_{m=1}^{M} I_{1,m} \subset K.
\end{eqnarray*}
\subsection{Some useful results}

In this section, we give two lemmas  that are useful for proving Theorem \ref{thr1}.
\begin{Lemma}\label{lem1}
We have:

\begin{itemize}
\item[(i)] $\sqrt{n}\left(\widehat{p}_{m}^{(n)} - p_{m} \right) = \pi_{3}^{m}(\widehat{W}^{(n)})$;
\item[(ii)]
$\sqrt{n}\left(\widehat{p}_{\ell|m}^{(n)} - p_{\ell|m} \right) = p_{m}^{-1}\left(\pi_{4}^{\ell m}\left(\widehat{W}^{(n)}\right) - \pi_{3}^{m}\left(\widehat{W}^{(n)}\right)\widehat{p}_{\ell|m}^{(n)}\right)$;
\item[(iii)]
$\sqrt{n}\left(\widehat{\mu}_{m}^{(n)} - \mu_{m} \right) =  p_{m}^{-1}\left(\pi_{2}^{m}(\widehat{W}^{(n)}) - \pi_{3}^{m}(\widehat{W}^{(n)})\widehat{\mu}_{m}^{(n)}\right)$; 
\item[(iv)]
$\sqrt{n}\left(\widehat{\mu}_{\ell,m}^{(n)} - \mu_{\ell,m} \right) = p_{m}^{-1} p_{\ell\vert m}^{-1}\left(\pi_{1}^{\ell m}(\widehat{W}^{(n)}) - \pi_{4}^{\ell m}(\widehat{W}^{(n)})\widehat{\mu}_{\ell, m}^{(n)}\right)$;.
\item[(v)]
\begin{eqnarray*}
\sqrt{n}\left(\widehat{V}_{m}^{(n)} - V_{m}\right) &=& p_{m}^{-1}\left[\pi_{5}^{m}\left(\widehat{W}^{(n)}\right) - \pi_{3}^{m}\left(\widehat{W}^{(n)}\right)(\widehat{V}_{m}^{(n)} + \widehat{\mu}_{m}^{(n)}\otimes \widehat{\mu}_{m}^{(n)})\right]\\
                 &-& p_{m}^{-1}\left[\left(\pi_{2}^{m}(\widehat{W}^{(n)}) - \pi_{3}^{m}(\widehat{W}^{(n)})\widehat{\mu}_{m}^{(n)}\right)\otimes \widehat{\mu}_{m}^{(n)} \right.\\
                                                   &+& \left. \mu_{m}\otimes \left(\pi_{2}^{m}(\widehat{W}^{(n)}) - \pi_{3}^{m}(\widehat{W}^{(n)})\widehat{\mu}_{m}^{(n)}\right)\right].
\end{eqnarray*}
\end{itemize}
\end{Lemma}
\begin{proof}

\noindent $ (i)$. Clearly,
\begin{eqnarray*}
 \pi_{3}^{m}\left(\widehat{W}^{(n)}\right)=\sqrt{n}\left(\frac{1}{n}\sum_{i=1}^{n}\textrm{\textbf{1}}_{\left\{U_i=m\right\}} - \E(\textrm{\textbf{1}}_{\left\{U=m\right\}})\right)
=\sqrt{n}\left(\widehat{p}_{m}^{(n)} - p_{m}\right).
\end{eqnarray*}

\noindent $(ii)$.  
\begin{eqnarray*}
 \pi_{4}^{\ell m}\left(\widehat{W}^{(n)}\right)&=&\sqrt{n}\left(\frac{1}{n}\sum_{i=1}^{n}\textrm{\textbf{1}}_{\left\{Z_i=\ell,\,U_i=m\right\}} - \E(\textrm{\textbf{1}}_{\left\{Z=\ell,\,U=m\right\}})\right)\\
&=&\sqrt{n}\left(\frac{\widehat{N}^{(n)}_{\ell,m}}{n}- P(Z=\ell,\,U=m)\right)\\
&=&\sqrt{n} \left(\widehat{p}_{m}^{(n)}\widehat{p}_{\ell|m}^{(n)} - p_{m}p_{\ell|m}\right).
\end{eqnarray*}
Further, 
\begin{eqnarray*}
\sqrt{n}\left(\widehat{p}_{m}^{(n)}\widehat{p}_{\ell|m}^{(n)} - p_{m}p_{\ell|m}\right) &=& \sqrt{n}\left(\widehat{p}_{m}^{(n)} - p_{m}\right)\widehat{p}_{\ell|m}^{(n)} + p_{m}\sqrt{n}\left(\widehat{p}_{\ell|m}^{(n)} - p_{\ell|m} \right)\\
&=&  \pi_{3}^{m}\left(\widehat{W}^{(n)}\right)\widehat{p}_{\ell|m}^{(n)} +p_{m}\sqrt{n}\left(\widehat{p}_{\ell|m}^{(n)} - p_{\ell|m} \right).
\end{eqnarray*}
Hence
\[
\sqrt{n}\left(\widehat{p}_{\ell|m}^{(n)} - p_{\ell|m}\right) = p_{m}^{-1}\left(\pi_{4}^{\ell m}\left(\widehat{W}^{(n)}\right) - \pi_{3}^{m}\left(\widehat{W}^{(n)}\right)\widehat{p}_{\ell|m}^{(n)}\right).
\]
$(iii)$.
\begin{eqnarray*}
 \pi_{2}^{ m}\left(\widehat{W}^{(n)}\right)&=&\sqrt{n}\left(\frac{1}{n}\sum_{i=1}^{n}\textrm{\textbf{1}}_{\left\{U_i=m\right\}}X_i - \E(\textrm{\textbf{1}}_{\left\{U=m\right\}}X)\right)= \sqrt{n}\left(\widehat{p}_{m}^{(n)}\widehat{\mu}_{m}^{(n)} - p_{m}\mu_{m}\right)\\
&=& \sqrt{n} \left( \widehat{p}_{m}^{(n)}- p_{m}\right) \widehat{\mu}_{m}^{(n)} + p_{m}\sqrt{n}\left(\widehat{\mu}_{m}^{(n)} - \mu_{m} \right)\\
&=& \pi_{3}^{m}\left(\widehat{W}^{(n)}\right)\widehat{\mu}_{m}^{(n)} + p_{m}\sqrt{n}\left(\widehat{\mu}_{m}^{(n)} - \mu_{m}) \right).
\end{eqnarray*}
Hence
\[
\sqrt{n}(\widehat{\mu}_{m}^{(n)} - \mu_{m}) = p_{m}^{-1}\left(\pi_{2}^{m}\left(\widehat{W}^{(n)}\right) - \pi_{3}^{m}\left(\widehat{W}^{(n)}\right)\widehat{\mu}_{m}^{(n)}\right).
\]
$(iv)$.
\begin{eqnarray*}
 \pi_{1}^{ \ell m}\left(\widehat{W}^{(n)}\right)&=&\sqrt{n}\left(\frac{1}{n}\sum_{i=1}^{n}\textrm{\textbf{1}}_{\left\{Z_i=\ell,\,U_i=m\right\}}X_i - \E(\textrm{\textbf{1}}_{\left\{Z=\ell,\,U=m\right\}}X)\right)\\
&=& \sqrt{n}\left(\widehat{p}_{\ell,m}^{(n)}\widehat{\mu}_{\ell,m}^{(n)} - p_{\ell,m}\mu_{\ell,m}\right),
\end{eqnarray*}
where 
\begin{equation}\label{plm}
\widehat{p}_{\ell,m}^{(n)}=\frac{\widehat{N}_{\ell,m}^{(n)}}{n}= \widehat{p}_{m}^{(n)}\widehat{p}_{\ell|m}^{(n)}\,\,\textrm{ and }p_{\ell,m}=P(Z=\ell,U=m)=p_mp_{\ell\vert m}.
\end{equation}
 Moreover,
\begin{eqnarray*}
\sqrt{n}\left(\widehat{p}_{\ell,m}^{(n)}\widehat{\mu}_{\ell,m}^{(n)} - p_{\ell,m}\mu_{\ell,m}\right)&=& \sqrt{n} \left( \widehat{p}_{\ell,m}^{(n)} - p_{\ell,m}\right)\widehat{\mu}_{\ell,m}^{(n)} + p_{\ell,m}\sqrt{n}\left(\widehat{\mu}_{\ell,m}^{(n)} - \mu_{\ell,m}\right)\\
&=&\pi_{4}^{ \ell m}\left(\widehat{W}^{(n)}\right)\widehat{\mu}_{\ell,m}^{(n)} + p_{\ell,m}\sqrt{n}\left(\widehat{\mu}_{\ell,m}^{(n)} - \mu_{\ell,m}\right).
\end{eqnarray*}
From this last equality and the second one in (\ref{plm}), we deduce that
\[
\sqrt{n}(\widehat{\mu}_{\ell,m}^{(n)} - \mu_{\ell,m}) = p_{m}^{-1} p_{\ell\vert m}^{-1}\left(\pi_{1}^{\ell m}\left(\widehat{W}^{(n)}\right) - \pi_{3}^{\ell m}\left(\widehat{W}^{(n)}\right)\widehat{\mu}_{\ell,m}^{(n)}\right).
\]
$(v)$.
\begin{eqnarray*}
 \pi_{5}^{ m}\left(\widehat{W}^{(n)}\right)&=&\sqrt{n}\left(\frac{1}{n}\sum_{i=1}^{n}\textrm{\textbf{1}}_{\left\{U_i=m\right\}}X_i\otimes X_i - \E(\textrm{\textbf{1}}_{\left\{U=m\right\}}X\otimes X)\right)\\
 &=& \sqrt{n}\left(\widehat{p}_{m}^{(n)}\frac{1}{\widehat{N}_{m}^{(n)}}\sum_{i=1}^{n}\textrm{\textbf{1}}_{\left\{U_i=m\right\}} X_{i}\otimes X_{i} - p_{m}\E(X\otimes X|U=m)\right)\\
 &=& \sqrt{n}\left(\widehat{p}_{m}^{(n)}\widehat{V}_{m}^{(n)} + \widehat{p}_{m}^{(n)}\widehat{\mu}_{m}^{(n)}\otimes \widehat{\mu}_{m}^{(n)} - p_{m}V_{m} - p_{m}\,\mu_{m}\otimes \mu_{m}\right)\\
                                          &=& p_{m}\sqrt{n}(\widehat{V}_{m}^{(n)} - V_{m}) + \sqrt{n}(\widehat{p}_{m}^{(n)} - p_{m})\left(\widehat{V}_{m}^{(n)} + \widehat{\mu}_{m}^{(n)}\otimes \widehat{\mu}_{m}^{(n)} \right)\\
                                          &+& p_{m}\left(\sqrt{n}(\widehat{\mu}_{m}^{(n)} - \mu_{m})\otimes \widehat{\mu}_{m}^{(n)} + \mu_{m}\otimes \sqrt{n}(\widehat{\mu}_{m}^{(n)} - \mu_{m}) \right)\\
 &=& p_{m}\sqrt{n}(\widehat{V}_{m}^{(n)} - V_{m}) + \pi_{3}^{ m}\left(\widehat{W}^{(n)}\right)\left(\widehat{V}_{m}^{(n)} + \widehat{\mu}_{m}^{(n)}\otimes \widehat{\mu}_{m}^{(n)} \right)\\
                                          &+& \left(\pi_{2}^{m}\left(\widehat{W}^{(n)}\right) - \pi_{3}^{m}\left(\widehat{W}^{(n)}\right)\widehat{\mu}_{m}^{(n)}\right)\otimes \widehat{\mu}_{m}^{(n)} \\
&+ & \mu_{m}\otimes \left(\pi_{2}^{m}\left(\widehat{W}^{(n)}\right) - \pi_{3}^{m}\left(\widehat{W}^{(n)}\right)\widehat{\mu}_{m}^{(n)}\right)
\end{eqnarray*}
Thus
\begin{eqnarray*}
\sqrt{n}\left(\widehat{V}_{m}^{(n)} - V_{m}\right)     
                                                    &=& p_{m}^{-1}\left[\pi_{5}^{m}\left(\widehat{W}^{(n)}\right) - \pi_{3}^{m}\left(\widehat{W}^{(n)}\right)(\widehat{V}_{m}^{(n)} + \widehat{\mu}_{m}^{(n)}\otimes \widehat{\mu}_{m}^{(n)})\right]\\
                                                   &-& p_{m}^{-1}\left[\left(\pi_{2}^{m}(\widehat{W}^{(n)}) - \pi_{3}^{m}(\widehat{W}^{(n)})\widehat{\mu}_{m}^{(n)}\right)\otimes \widehat{\mu}_{m}^{(n)}\right. \\
                                                   &+& \left. \mu_{m}\otimes \left(\pi_{2}^{m}(\widehat{W}^{(n)}) - \pi_{3}^{m}(\widehat{W}^{(n)})\widehat{\mu}_{m}^{(n)}\right)\right].
\end{eqnarray*}
\end{proof}


\begin{Lemma}\label{lem2}
For any  $K \subset I$ and any  $m \in \left\{1,\cdots,M\right\}$:
\begin{enumerate}
	\item[(i)] $\widehat{\xi}_{K|m}^{(n)}$ converges almost surely to  $\xi_{K|m}$ as $n \rightarrow +\infty$;
	\item[(ii)] $n \widehat{\xi}_{K|m}^{(n)}$ = $\sum_{\ell=1}^{q} \left(\widehat{\Phi}_{\ell,K\vert m}^{(n)}(\widehat{W}^{(n)}) + p_{\ell|m}\|\widehat{\Psi}_{\ell,K\vert m}^{(n)}(\widehat{W}^{(n)}) + \sqrt{n}\Delta_{\ell,K\vert m}\|  \right)^{2}$.
\end{enumerate}
\end{Lemma}

\begin{proof}

\noindent $(i)$. First, using the strong law of large numbers (SLLN) we obtain the almost sure convergence of $\widehat{N}^{(n)}_m/n$ (resp. $\widehat{N}^{(n)}_{\ell ,m}/n$)  to $p_m$ (resp. $p_{\ell ,m}=p_mp_{\ell\vert m}$) as $n\rightarrow +\infty$. Then, as $n\rightarrow +\infty$, $\widehat{p}^{(n)}_{\ell\vert m}$ converges almost surely (a.s.)  to $p_{\ell\vert m}$ and, by the preceding convergence properties and the SLLN,
\[
\widehat{\mu}_{m}^{(n)} = \left(\frac{\widehat{N}_{ m}^{(n)}}{n}\right)^{-1}\left(\frac{1}{n}\sum_{i=1}^{n} \textrm{\textbf{1}}_{\left\{U_i=m\right\}}X_{i} \right)\stackrel{a.s.}{\longrightarrow}p_m^{-1}\mathbb{E}\left(\textrm{\textbf{1}}_{\left\{U=m\right\}}X\right)=\mu_m,
\]
\[
\widehat{\mu}_{\ell,m}^{(n)} = \left(\frac{\widehat{N}_{\ell , m}^{(n)}}{n}\right)^{-1}\left(\frac{1}{n}\sum_{i=1}^{n} \textrm{\textbf{1}}_{\left\{Z_{i}=\ell,U_i=m\right\}}X_{i} \right)\stackrel{a.s.}{\longrightarrow}p_{\ell ,m}^{-1}\mathbb{E}\left(\textrm{\textbf{1}}_{\left\{Z=\ell ,U=m\right\}}X\right)=\mu_{\ell , m}
\]
and
\begin{eqnarray*}
\widehat{V}^{(n)}_m&=& \left(\frac{\widehat{N}_{ m}^{(n)}}{n}\right)^{-1}\left(\frac{1}{n}\sum_{i=1}^{n} \textrm{\textbf{1}}_{\left\{U_i=m\right\}}X_{i}\otimes X_i\right)-\widehat{\mu}^{(n)}\otimes \widehat{\mu}^{(n)}\\
& &\stackrel{a.s.}{\longrightarrow}p_m^{-1}\mathbb{E}\left(\textrm{\textbf{1}}_{\left\{U=m\right\}}X\otimes X\right)-\mu\otimes\mu=V_m.
\end{eqnarray*}
Hence, $\widehat{Q}^{(n)}_{K\vert m}$  converges almost surely  to $Q_{K\vert m}$  as $n\rightarrow +\infty$  and, finally, we obtain from all these results the required almost sure convergence of $\xi^{(n)}_{K\vert m}$ to $\xi_{K\vert m}$ as $n\rightarrow +\infty$.

\noindent $(ii)$. Putting
$\widehat{\Delta}^{(n)}_{\ell,K\vert m} = \left(I_{\R^{p}} - \widehat{V}^{(n)}_{m} \widehat{Q}^{(n)}_{K\vert m}\right)\left(\widehat{\mu}^{(n)}_{\ell,m} - \widehat{\mu}^{(n)}_{m}\right)$,  
we have
\begin{eqnarray*}
\sqrt{n}\widehat{\Delta}^{(n)}_{\ell,K\vert m} &=& \left(I_{\R^{p}} - V_{m}Q_{K\vert m}\right)\left[\sqrt{n}\left(\widehat{\mu}^{(n)}_{\ell,m} - \mu_{\ell,m}\right) - \sqrt{n}\left(\widehat{\mu}^{(n)}_{m} - \mu_{m}\right) \right.\\
&-& \left. \left(\sqrt{n}\left(\widehat{V}^{(n)}_{m} - V_{m}\right)\right)\widehat{Q}^{(n)}_{K\vert m}\left(\widehat{\mu}^{(n)}_{\ell,m} - \widehat{\mu}^{(n)}_{m}\right)\right]
+ \sqrt{n}\Delta_{\ell,K\vert m}.
\end{eqnarray*}
Then, using Lemma \ref{lem1},  we obtain $\sqrt{n}\widehat{\Delta}^{(n)}_{\ell,K\vert m} =\widehat{\Psi}_{\ell,K\vert m}^{(n)}\left(\widehat{W}^{(n)}\right) + \sqrt{n}\Delta_{\ell,K\vert m}$, where $\widehat{\Psi}_{\ell,K\vert m}^{(n)}$ is the random operator defined on $\mathcal{E}$ by
\begin{eqnarray*}
\widehat{\Psi}_{\ell,K|m}\left(T\right) &=&p_m^{-1} \left(I_{\R^{p}} - V_{m}Q_{K\vert m}\right)
 \left[p_{\ell|m}^{-1}\left(\pi_{1}^{\ell m}(T) - \pi_{4}^{\ell m}(T)\widehat{\mu}_{\ell,m}^{(n)}\right)\right.
 - \pi_{2}^{m}(T) + \pi_{3}^{m}(T)\widehat{\mu}_{m}^{(n)}\\
& & -\left(\pi_{5}^{m}(T) - \pi_{3}^{m}(T)\left(\widehat{V}^{(n)}_{m} + \widehat{\mu}_{m}^{(n)}\otimes \widehat{\mu}_{m}^{(n)}\right)\right)\widehat{Q}^{(n)}_{K\vert m}\left(\mu^{(n)}_{\ell,m} - \widehat{\mu}_{m}^{(n)}\right)\\
& & +\left(\left(\pi_{2}^{m}(T) - \pi_{3}^{m}(T) \widehat{\mu}_{m}^{(n)}\right)\otimes\widehat{\mu}_{m}^{(n)} \right) \widehat{Q}^{(n)}_{K\vert m}\left(\mu^{(n)}_{\ell,m} - \widehat{\mu}_{m}^{(n)}\right)\\
& &+ \left.\left(\mu_{m}\otimes \left(\pi_{2}^{m}(T) - \pi_{3}^{m}(T)\widehat{\mu}_{m}^{(n)}\right)\right)\widehat{Q}^{(n)}_{K\vert m}\left(\mu^{(n)}_{\ell,m} - \widehat{\mu}_{m}^{(n)}\right) \right].
\end{eqnarray*}
Since
\begin{eqnarray*}
n\widehat{\xi}^{(n)}_{K|m} &=& 
\sum_{\ell=1}^{q}\left(\sqrt{n}\left(\widehat{p}^{(n)}_{\ell|m} - p_{\ell|m}\right)\|\widehat{\Delta}^{(n)}_{\ell,K\vert m} 
\|_{\R^{p}} + p_{\ell|m}\|\sqrt{n}\widehat{\Delta}^{(n)}_{\ell,K,m}\|\right)^{2},
\end{eqnarray*}
we deduce from what precedes and Lemma \ref{lem1} that 
\begin{eqnarray*}
n\widehat{\xi}^{(n)}_{K|m}
                           &=& \sum_{\ell=1}^{q}\left(\widehat{\Phi}_{\ell,K,m}^{(n)}\left(\widehat{W}^{(n)}\right) + p_{\ell|m}\|\widehat{\Psi}_{\ell,K,m}^{(n)}\left(\widehat{W}^{(n)}\right) + \sqrt{n}\Delta_{\ell,K\vert m}\| \right)^{2}
\end{eqnarray*}
where $\widehat{\Phi}_{\ell,K\vert m}^{(n)}$ is the random operator defined on $\mathcal{E}$ by
\[
\widehat{\Phi}_{\ell,K,m}^{(n)}\left(T\right) = p_{m}^{-1}\left(\pi_{4}^{\ell m}\left(T\right) - \pi_{3}^{m}\left(T\right)\widehat{p}_{\ell|m}^{(n)}\right)\|\widehat{\Delta}^{(n)}_{\ell,K\vert m}.
\|
\]
The almost sure convergences obtained in the proof of $(i)$ imply that  $\widehat{\Phi}^{(n)}_{\ell,K\vert m}$ (resp.$\widehat{\Psi}^{(n)}_{\ell,K\vert m}$) converges almost surely uniformly to $\Phi_{\ell,K\vert m}$ (resp.$\Psi_{\ell,K\vert m}$) as $n \rightarrow +\infty$. 
\end{proof}
\subsection{Proof of Theorem \ref{thr1}}
Using  Assertion $(i)$ of Lemma \ref{lem2} we obviously obtain $(i)$. For proving $(ii)$, note that
	$$\left(\widehat{p}_{m}^{(n)}\right)^{2}\widehat{\xi}_{K|m}^{(n)}=\left(\widehat{p}_{m}^{(n)} - p_{m}\right)\left(\widehat{p}_{m}^{(n)} + p_{m}\right)\widehat{\xi}_{K|m}^{(n)} + p_{m}^{2}\widehat{\xi}_{K|m}^{(n)}.$$ Then, using Lemma \ref{lem1} and Lemma \ref{lem2}, we obtain
\begin{eqnarray*}
n\left(\widehat{p}_{m}^{(n)}\right)^{2}\widehat{\xi}_{K|m}^{(n)}
                                                     = \sqrt{n}\widehat{\Lambda}_{K\vert m}^{(n)}(\widehat{W}^{(n)}) &+& \sum_{\ell=1}^{q}\left(p_{m}\widehat{\Phi}_{\ell,K\vert m}^{(n)}\left(\widehat{W}^{(n)}\right) \right.\\
& &\left.+
 p_{m}p_{\ell|m}\|\widehat{\Psi}_{\ell,K\vert m}^{(n)}\left(\widehat{W}^{(n)}\right) \right.
                                                     +\left. \sqrt{n}\Delta_{\ell,K\vert m}\| \right)^{2},
\end{eqnarray*}
where $\widehat{\Lambda}_{K\vert m}^{(n)}$ is the random operator defined on $\mathcal{E}$  by  
\[
\widehat{\Lambda}_{K\vert m}^{(n)}(T) = \pi_{3}^{m}(T)\left(\widehat{p}_{m}^{(n)} + p_{m} \right)\widehat{\xi}^{(n)}_{K|m}. 
\]
Then from (\ref{estimcrit}), it follows:
\begin{eqnarray*}
  n \widehat{\xi}_{K}^{(n)} &=& \sum_{m=1}^{M}\sqrt{n}\widehat{\Lambda}_{K|m}^{(n)}(\widehat{W}^{(n)}) + \sum_{m=1}^{M}\sum_{\ell=1}^{q}\left(p_{m}\widehat{\Phi}_{\ell,K|m}^{(n)}\left(\widehat{W}^{(n)}\right)  \right.\\
                            &+& \left.
p_{m}p_{\ell|m}\|\widehat{\Psi}_{\ell,K|m}^{(n)}\left(\widehat{W}^{(n)}\right) + \sqrt{n}\Delta_{\ell,K\vert m}\| \right)^{2}.
\end{eqnarray*}
The almost sure convergence of  $\widehat{p}_{m}^{(n)}$  to $p_{m}$,  as $n \rightarrow +\infty$, and Assertion $(i)$ of Lemma \ref{lem2} imply that $\widehat{\Lambda}_{K|m}^{(n)}$ converges almost surely uniformly to  $\Lambda_{K|m}$, as $n\rightarrow +\infty$.
\subsection{Proof of  Proposition \ref{prop2}}
\label{thcons}
From  Theorem \ref{thr1}, we obtain:
\begin{equation}\label{decompo}
n^{\alpha}\left(\widehat{\xi}^{(n)}_{K_{\sigma(i)}} - \widehat{\xi}^{(n)}_{K_{\sigma(i+1)}}\right) =\widehat{A}^{(n)}_{i}+\widehat{B}^{(n)}_{i}+\widehat{C}^{(n)}_{i}+\widehat{D}^{(n)}_{i} 
\end{equation}
where
\begin{eqnarray*}
\widehat{A}^{(n)}_{i} &=& \sum_{m=1}^{M}n^{\alpha -1/2}\left(\widehat{\Lambda}^{(n)}_{K_{\sigma(i)}\vert m}(\widehat{W}^{(n)}) - \widehat{\Lambda}^{(n)}_{K_{\sigma(i+1)}\vert m}(\widehat{W}^{(n)})\right),\\
\widehat{B}^{(n)}_{i} &=& \sum_{m=1}^{M}\sum_{\ell=1}^{q}p_{m}^{2}n^{\alpha -1}\left(\widehat{\Phi}^{(n)}_{\ell,K_{\sigma(i)}\vert m}(\widehat{W}^{(n)})^{2} - \widehat{\Phi}^{(n)}_{\ell,K_{\sigma(i+1)}\vert m}(\widehat{W}^{(n)})^{2} \right),\\
\widehat{C}^{(n)}_{i} &=& \sum_{m=1}^{M}\sum_{\ell=1}^{q}p_{m}^{2}p_{\ell|m}^{2}n^{\alpha -1} \left(\|\widehat{\Psi}^{(n)}_{\ell,K_{\sigma(i)}\vert m}(\widehat{W}^{(n)}) + \sqrt{n}\Delta_{\ell,K_{\sigma(i)}\vert m}\|^{2}\right.\\
& &-\left. \|\widehat{\Psi}^{(n)}_{\ell,K_{\sigma(i+1)}\vert m}(\widehat{W}^{(n)}) 
+ \sqrt{n}\Delta_{\ell,K_{\sigma(i+1)}\vert m}\|^{2}\right), \\
\widehat{D}^{(n)}_{i} &=& \sum_{m=1}^{M}\sum_{\ell=1}^{q}2p_{m}^{2}p_{\ell|m} n^{\alpha -1}\left(\widehat{\Phi}^{(n)}_{\ell,K_{\sigma(i)}\vert m}(\widehat{W}^{(n)})\|\widehat{\Psi}^{(n)}_{\ell,K_{\sigma(i)}\vert m}(\widehat{W}^{(n)}) + \sqrt{n}\Delta_{\ell,K_{\sigma(i)}\vert m}\| \right.\\
& &- \left. \widehat{\Phi}^{(n)}_{\ell,K_{\sigma(i+1)}\vert m}(\widehat{W}^{(n)})\|\widehat{\Psi}^{(n)}_{\ell,K_{\sigma(i+1)}\vert m}(\widehat{W}^{(n)}) + \sqrt{n}\Delta_{\ell,K_{\sigma(i+1)}\vert m}\|\right).
\end{eqnarray*}
Denoting by  $\|.\|_{\infty}$ the usual uniform convergence norm defined by$\|T\|_{\infty} = \sup_{x\neq0} \|Tx\|/\|x\|$, and by $\Vert\cdot\Vert_\mathcal{E}$ the norm of $\mathcal{E}$, we have
\begin{equation}\label{ai}
|\widehat{A}^{(n)}_{i}| 
                    \leq n^{\alpha -1/2} \left(\|\widehat{\Lambda}^{(n)}_{K_{\sigma(i)}\vert m}\|_{\infty} + \|\widehat{\Lambda}^{(n)}_{K_{\sigma(i+1)}\vert m}\|_{\infty}\right)\|\widehat{W}^{(n)}\|_\mathcal{E},
\end{equation}
\begin{equation}\label{bi}
\vert\widehat{B}^{(n)}_{i} \vert
                   \leq n^{\alpha -1}\left(\|\widehat{\Phi}^{(n)}_{\ell,K_{\sigma(i)}\vert m}\|_{\infty}^{2} + \|\widehat{\Phi}^{(n)}_{\ell,K_{\sigma(i+1)}\vert m}\|_{\infty}^{2}\right)\|\widehat{W}^{(n)}\|^{2}_\mathcal{E},
\end{equation}
\begin{eqnarray}\label{di}
|\widehat{D}_{i}^{(n)}| 
                       &\leq&
n^{\alpha-1/2}\left\|\widehat{W}^{(n)}\right\|_\mathcal{E}\left(\left\|\widehat{\Phi}^{(n)}_{\ell,K_{\sigma(i)}\vert m}\right\|_{\infty}\left(n^{-1/2}\|\widehat{\Psi}^{(n)}_{\ell,K_{\sigma(i)}\vert m}\|_{\infty}\|\widehat{W}^{(n)}\|_\mathcal{E} + \|\Delta_{\ell,K_{\sigma(i)}\vert m}\|\right) \right.\nonumber\\
& & + \,\,\, \left. \left\|\widehat{\Phi}^{(n)}_{\ell,K_{\sigma(i+1)}\vert m}\right\|_{\infty}\left(n^{-1/2}\|\widehat{\Psi}^{(n)}_{\ell,K_{\sigma(i+1)}\vert m}\|_{\infty}\|\widehat{W}^{(n)}\|_\mathcal{E} + \|\Delta_{\ell,K_{\sigma(i+1)}\vert m}\|\right)\right).                                      
\end{eqnarray}
Since $i\in F_\ell$, we have 
\[
 \|\Delta_{\ell,K_{\sigma(i)}\vert m}\|=\xi_{K_{\sigma(i)}}=\xi_{K_{\sigma(i+1)}}= \|\Delta_{\ell,K_{\sigma(i+1)}\vert m}\|
\]
and, therefore,
\begin{eqnarray*}
& &\|\widehat{\Psi}^{(n)}_{\ell,K_{\sigma(i)}\vert m}(\widehat{W}^{(n)}) + \sqrt{n}\Delta_{\ell,K_{\sigma(i)}\vert m}\|^{2}
- \|\widehat{\Psi}^{(n)}_{\ell,K_{\sigma(i+1)}\vert m}(\widehat{W}^{(n)}) 
+ \sqrt{n}\Delta_{\ell,K_{\sigma(i+1)}\vert m}\|^{2}\\
&=&\|\widehat{\Psi}^{(n)}_{\ell,K_{\sigma(i)}\vert m}(\widehat{W}^{(n)}) \|^{2}
- \|\widehat{\Psi}^{(n)}_{\ell,K_{\sigma(i+1)}\vert m}(\widehat{W}^{(n)}) 
\|^{2}+2\sqrt{n}<\widehat{\Psi}^{(n)}_{\ell,K_{\sigma(i)}\vert m}(\widehat{W}^{(n)}),\Delta_{\ell,K_{\sigma(i)}\vert m}>\\
&-&2\sqrt{n}<\widehat{\Psi}^{(n)}_{\ell,K_{\sigma(i+1)}\vert m}(\widehat{W}^{(n)}),\Delta_{\ell,K_{\sigma(i+1)}\vert m}>.
\end{eqnarray*}
Hence
\begin{eqnarray}\label{ci}
\vert\widehat{C}_{i}^{(n)}\vert
                        &\leq& n^{\alpha-1/2} \|\widehat{W}^{(n)}\|_\mathcal{E}\left(n^{-1/2}\left(\|\widehat{\Psi}^{(n)}_{\ell,K_{\sigma(i)}\vert m}\|_{\infty}^2 + \|\widehat{\Psi}^{(n)}_{\ell,K_{\sigma(i+1)}\vert m}\|_{\infty}^2\right)\|\widehat{W}^{(n)}\|_\mathcal{E}\right.\nonumber\\
& + & \left. 2\|\widehat{\Psi}^{(n)}_{\ell,K_{\sigma(i)}\vert m}\|_{\infty}\|\Delta_{\ell,K_{\sigma(i)}\vert m}\|+2\|\widehat{\Psi}^{(n)}_{\ell,K_{\sigma(i+1)}\vert m}\|_{\infty} \|\Delta_{\ell,K_{\sigma(i+1)},m}\|\right).
\end{eqnarray}  
From (\ref{wn}) and the central limit theorem, $ \widehat{W}^{(n)}$ converges in distribution, as $n\rightarrow +\infty$, to a random variable having a normal distribution in $\mathcal{E}$. Then, the almost sure  uniform convergences of $ \widehat{\Lambda}^{(n)}_{K_{\sigma(j)}\vert m}$,  $\widehat{\Phi}^{(n)}_{\ell,K_{\sigma(j)}\vert m}$ and $\widehat{\Psi}^{(n)}_{\ell,K_{\sigma(j)}\vert m}$  ($j\in I$), and the inequality $\alpha<1/2$ permit to deduce from (\ref{ai}),  (\ref{bi}),  (\ref{di}) and  (\ref{ci}) that $\widehat{A}_{i}^{(n)}$, $\widehat{B}_{i}^{(n)}$, $\widehat{C}_{i}^{(n)}$ and $\widehat{D}_{i}^{(n)}$ converge in probability to $0$, as $n\rightarrow +\infty$. Finally, the required result is obtained from (\ref{decompo}).

\end{document}